\documentclass[twoside,openright,english,12pt]{amsart}
\usepackage[T1]{fontenc}
\usepackage{lmodern}
\usepackage{geometry} 
\usepackage{babel}
\usepackage{amssymb,latexsym,amsmath,amsthm} 
\usepackage{color}
\usepackage{enumerate}

\newtheorem{theorem}{Theorem}[section]
\newtheorem{lemma}[theorem]{Lemma}

\newtheorem{proposition}[theorem]{Proposition}
\theoremstyle{definition}
\newtheorem{definition}[theorem]{Definition}

\newtheorem{remark}{Remark}[section]

\def\C{\mathbb{C}}
\def\R{\mathbb{R}}
\def\H{\mathbb{H}}
\def\O{\mathbb{O}}
\def\Z{\mathbb{Z}}

\def\bv{\mathbf{v}}
\def\bw{\mathbf{w}}
\def\bx{\mathbf{x}}
\def\be{\mathbf{e}}
\def\bE{\mathbf{E}}
\def\bV{\mathbf{V}}

\def\cC{\mathcal{C}}
\def\cF{\mathcal{F}}
\def\cH{\mathcal{H}}
\def\cI{\mathcal{I}}
\def\cO{\mathcal{O}}
\def\cP{\mathcal{P}}
\def\cF{\mathcal{F}}
\def\cT{\mathcal{T}}

\def\cW{\mathcal{W}}

\def\fC{\mathfrak{C}}

\def\fg{\mathfrak{g}}
\def\fG{\mathfrak{G}}
\def\fK{\mathfrak{K}}
\def\fM{\mathfrak{M}}

\def\e{\varepsilon}

\title{$F$--MANIFOLDS AND GEOMETRY OF INFORMATION}

\author{No\'emie Combe}
\address{No\'emie Combe, Max Planck Institut for Matthematics in Sciences, Inselstra\ss e 22, 04103 Leipzig, Germany}
\email{noemie.combe@mis.mpg.de}

\author{Yuri I. Manin}
\address{Yuri I. Manin, Max Planck Institut for Mathematics, Vivatsgasse 7, 53111 Bonn, Germany}
\email{manin@mpim-bonn.mpg.de}
\date{02-02-2020 }
\thanks{Both authors wish to thank the anonymous referee for suggestions. The first author wishes to thank Professor Ph. Combe and Professor H. Nencka for many helpful discussions and suggestions concering statistical manifolds. The first author is grateful to the Minerva grant from the Max Planck Society for support. The second author thanks the Max Planck Institute for support.}
\begin{document}
\maketitle
\begin{abstract}
The theory of $F$--manifolds, and more generally, manifolds endowed with
commutative and associative multiplication of their tangent fields, was discovered
and formalised in various models of quantum field theory
involving algebraic and analytic geometry, at least since 1990's.

The focus of this paper consists in the demonstration that
various spaces of probability distributions defined and studied
at least since 1960's also carry
natural structures of $F$--manifolds.

This fact remained somewhat hidden in various domains of the vast territory
of models of information storing and transmission that are briefly
surveyed here.
\end{abstract}

\setcounter{tocdepth}{1}
\tableofcontents



\section{Introduction and summary}\label{S:intro}

The structure of {\it Frobenius manifolds} and its later weakened versions
{\it weak Frobenius manifolds}, also called {\it $F$--manifolds}, was discovered in the 1980's and 1990's in the process
of development and formalisation of Topological Field Theory, including Mirror
Conjecture: see~\cite{D96},\cite{HeMa99}, and references therein.

Below, speaking about {\it (super)manifolds $M$}, we understand objects
of one of the standard geometric categories: $C^{\infty}$, analytic, algebraic,
formal etc.

 According to B. Dubrovin (\cite{D96} and \cite{Ma99}), the main component of a Frobenius structure
 on $M$ is a {\it (super)commutative, associative and bilinear over constants
 multiplication $\circ : \cT_M \otimes \cT_M \to \cT_M$ on its tangent sheaf $\cT_M$}.
 
 Additional parts of the structure in terms of which further restrictions upon $\circ$
 might be given, are listed below:
 \begin{list}{--}{}
\item A subsheaf of flat vector fields $\cT^f_M \subset \cT_M$ consisting of
 tangent vectors flat in a certain affine structure.
 
\item A metric (nondegenerate symmetric quadratic form) $g: S^2(\cT_M)\to \cO_M$.
 
\item An identity $e$. 
 
\item An Euler vector field $E$ .
\end{list}

\vspace{5pt}
 
 {\it Relationships} between/{\it restrictions} upon all these structures depend on
 the context in which they appeared in various research domains. Accordingly,
 the versions of structures themselves were called by various names:
 besides Frobenius manifolds and $F$--manifolds, the reader can find
 {\it pre--Frobenius, weak Frobenius} (\cite{Ma99}), and most recently, 
 {\it Frobenius--like structures of order $(n,k,m)$} (in the latter, the tangent sheaf
 is replaced by an external sheaf,~\cite{HeVa18}). Therefore, we will not be very
 strict with terminology.

\vspace{5pt}

The popularity of Frobenius manifolds among algebraic/analytic geometers
 was growing after initial discovery
of three large classes of them, naturally arising in mathematics and physics:

\begin{enumerate}[(i)]

\item A choice of Saito's {\it good primitive form} determines a natural Frobenius
structure upon moduli (unfolding) spaces of germs of isolated singularities
of hypersurfaces (topological sector of the Landau--Ginzburg theory for physicists):
see~\cite{Sa82},\cite{Sa83},\cite{Od85}.

\vspace{3pt}

\item The formal moduli spaces of solutions to the Maurer--Cartan equations 
modulo gauge equivalence have natural formal Frobenius structure,
if these Maurer--Cartan equations are stated in the dGBV
(differential Gerstenhaber--Batalin--Vilko\-visky) framework:
see~\cite{BaKo72} and~\cite{LiZu93}.

\vspace{3pt}

\item The formal completion at zero of the cohomology (super)space of
any smooth projective (or compact symplectic) manifold
carries a natural formal Frobenius supermanifold (theory of
Gromov--Witten invariants): for early mathematical sources
see~\cite{KoMa94},\cite{Beh97},\cite{BehMa96}.

\vspace{3pt}

Here we add to this list

\vspace{3pt}
\item Convex homogeneous cones (\cite{Vi63},\cite{Vi65},\cite{BeIo78})
and the spaces of probability distributions (see \cite{BuCoNe00} and the monographs~\cite{Ch82},\cite{Am85}, \cite{AmNa00}).
\end{enumerate}
\vspace{5pt}

We show that, under some restrictions, these spaces carry structures of $F$--manifolds.
\vspace{3pt}

In fact, on the map of ``Information Geometry Land'' there is another domain,
connecting this land {\it not} with rigid structures such as metric geometries
(which are in the focus of this paper), but with rather more fluid ones,
of homological and especially homotopical algebra. Intuitively,
one starts with imagining, say, configurations of neural nets in brain as 
simplicial complexes, accompanied by highly non--obvious heuristic observation
that {\it complexity of information} that can be successfully treated by such a net
grows with {\it complexity of homotopy class} of its geometric
realisation. 
\vspace{3pt}

For an introduction to this domain aimed to mathematicians, cf.~\cite{Mar19}
and \cite{MaMar20}.
\vspace{3pt}

The contents of our survey are distributed as follows.
\vspace{3pt}

Section~\ref{S:Fman} of this article contains a survey of geometry of
Frobenius--like manifolds. 
\vspace{3pt}

In Section~\ref{S:sing}, we focus on the appearance of this structure
on the unfolding spaces of isolated singularities,
and stress the role of so called {\it potentials}
that reappear further in information geometry.
\vspace{3pt}

 Section~\ref{S:covcone} introduces $F$--structures
upon convex homogeneous cones and spaces of probability distributions
stressing the environment in which these $F$--structures look
similar to the ones of previous Section, but with replacement
of framework of complex varieties with the one of real geometry. 
\vspace{3pt}

Finally, Section~\ref{S:statman} introduces ``paracomplex'' structures bridging
complex and real geometry in this context, and revealing paracomplex
potentials.
\vspace{3pt}


\section{Frobenius manifolds and $F$--manifolds}\label{S:Fman}

\subsection{Frobenius manifolds.} We start, as above,
with a family of data 
\begin{equation}\label{E:FM}
(M; \quad \circ : \cT_M\otimes \cT_M \to \cT_M; \quad \cT_M^f \subset \cT_M; \quad
g : S^2(\cT_M) \to \cO_M), 
\end{equation}
mostly omitting identity $e$ and Euler field $E$.
\vspace{3pt}

The main additional structure bridging these data together is a family of (local) {\it potentials} $\Phi$
(sections of $\cO_M$)
such that for any (local) flat tangent fields $X,Y,Z$ we have
\[
g(X\circ Y,Z) = g(X, Y\circ Z) =(XYZ)\Phi .
\]
If such a structure exists, then (super)commutativity and associativity of $\circ$
follows automatically, and we say that the family~\eqref{E:FM} defines a {\it Frobenius manifold.}
\vspace{3pt}

\subsection{${F}$--identity}
\

This identity relates multiplication $\circ$ and the Lie
(super)commutator that in the theory of Frobenius manifolds
follows from the basic definitions, and in the theory
of $F$--manifolds is postulated.
\vspace{3pt}

It is convenient to introduce first the auxiliary {\it Poisson tensor} 
$P: \cT_M\times \cT_M \times \cT_M \to \cT_M$
\[
P_X (Z,W) := [X,Z\circ W] - [X,Z]\circ W - (-1)^{XZ} Z\circ [X,W] .
\]

Here and further on we write $(-1)^{XZ}$ in place of $(-1)^{|X||Z|}$, where $|X|$
denotes the parity ($\Z_2$--degree) of $X$.
\vspace{3pt}

About relationship between Poisson tensors and manifolds with Poisson
structure, (cf.~\cite{Ma99}, subsection 5.5, p. 47, and~\cite{Ma19}, Sec. 5).
\vspace{3pt}

\begin{definition}~\label{D:Fmanifold}
 Let $M$ be a (super)manifold endowed with
(super)commutative and associative multiplication $\circ$ in its
tangent sheaf.
\vspace{3pt}

$M$ with this structure is called an $F$--manifold, if it satisfies
the $F$--identity:
\[
P_{X\circ Y} = X\circ P_Y(Z,W) + (-1)^{XY} Y\circ P_X(Z,W).
\]
\end{definition}
\vspace{3pt}

\subsection{Compatible flat structures}~\label{S:cft}
\

An {\it affine flat structure} on a manifold $M$, by definition, is a local system
$\cT^f_M \subset \cT_M$ of finite--dimensional (over constants) supercommutative Lie algebras
of rank $dim\,M$ such that $\cT_M= \cO_M\otimes \cT^f_M$.
\vspace{3pt}

In the situation of 1.1, but not postulating
$F$--identity, assume that in a neighbourhood of any point of $M$
there exists a vector field $C$ such that the $\circ$--product of arbitrary
local flat fields $X,Y$ defined in this neighbourhood can be written as
\[
X\circ Y= [X,[Y,C]] .
\]
Such $C$ is called a local {\it vector potential} for $\circ$. Then we will call
$\cT_M^f$ {\it compatible} with $\circ$. If $\circ$ admits a {\it flat} identity $e$,
we will call $\cT_M^f$ compatible with $(\circ , e)$.
\vspace{3pt}

\begin{proposition}~\label{P:fid} 
In the situation of Definition~\ref{D:Fmanifold}, if $\circ$ admits a compatible
flat structure, then it satisfies the $F$--identity. Thus, $(M, \cT_M, \circ )$
is an $F$--manifold.
\end{proposition}
\vspace{3pt}

In the context of geometry of statistics/information, the following equivalent descriptions of
flat structure might be useful

\begin{list}{--}{}

\item An atlas of local coordinates, whose transition functions are affine
linear (over constants).
\vspace{3pt}

\item A torsionless flat connection $\nabla_0 : \cT_M \to \Omega^1_M \otimes_{\cO_M} \cT_M$.
\end{list}

Indeed, given $\nabla_0$, we can define $\cT_M^f$ as $Ker\, \nabla_0$. 
\vspace{3pt}


\section{ $F$--manifolds and singularities}~\label{S:sing}

\subsection{K. Saito's frameworks}~\label{S:saito}
\

We will describe here in considerable detail a class of (pre--)Frobenius structures
that was introduced by K. Saito in the context of unfolding isolated
singularities and periods of primitive forms (see~\cite{Sa82},\cite{Sa83},\cite{Ma98}). Our choice is motivated by
the fact that the central objects of the next Sec.~\ref{S:covcone} coming from a very different
environment (convex cones and probability spaces) look strikingly similar 
to real versions of Saito's frameworks.
 \vspace{3pt}

Intuitively, Saito's $F$--structures are canonical data arising upon 
unfolding spaces of isolated singularities, both in analytic and algebraic
geometry, in characteristic zero.
 \vspace{3pt}

More precisely,  let $p: N\to M$ be a submersive morphism of complex analytic (or algebraic) varieties
(we do not assume them to be compact). Denote by $d_p: \cO_M \to \Omega^1_{N/M}$
its relative differential. For a holomorphic function $F$ on $M$, the equation
$d_pF = 0$ defines the closed analytic subspace $i_C: C=C_{N/M}(F) \hookrightarrow N$ of fibrewise
critical points of $F$; denote by $p_C: C\to M$ the restriction of $p$ to $C$.
We will need also the invertible sheaf of holomorphic vertical volume
forms $\Omega^{max}_N/M$ and its restriction $L:= i_C^* (\Omega^{max}_{N/M})$ to $C$. 
Finally, we will assume given a nowhere vanishing global section $\omega$ of $\Omega_{N/M}^{\max}$.
 \vspace{3pt}

The following Definition and Proposition (due to K.~Saito) are borrowed from~(\cite{Ma98}, 2.1.1).

\begin{definition}\label{D:sf}
 The family of data $(p: N\to M; F; \omega )$ as above is called Saito's framework
if it satisfies the following additional conditions.
 \vspace{3pt}
 \begin{enumerate}
\item Define the map $s:\cT_M \to p_{C*}(\cO_C)$ by $X \mapsto \overline{X} F\ \mathrm{mod}\, J_F$
where $J_F$ is the ideal defining $C$. Assume that $C$ is finite and flat over M.
 \vspace{3pt}
 
\item Now consider the Hessian of function $F$. In local coordinates $z=(z_a)$, $a=1, \dots , m$; $t = (t_b)$ such that $(t_b)$
is a maximal set of coordinates constant along fibres of $p$ the Hessian can be defined
as a section of $L^2$ that can be written as
\[
\mathrm{Hess} (F) := i_C^* [\det (\partial^2F/\partial z_a \partial z_b) (dz_1\wedge \dots \wedge dz_n)^2] .
\]
\end{enumerate}
\end{definition}
Make an additional assumption that the subspace $G_C$ of zeroes of $\mathrm{Hess} (F)$ is a divisor,
and that $p_C$ is \'etale outside the divisor $G:= i_{C*}(G_C)$ in $M$. 
 \vspace{3pt}
 
\begin{proposition} Let $(p: N\to M; F; \omega)$ be a Saito's framework. Consider a local tangent
field $X$ on $M$ over whose definition domain $p_{C*}$ is a disjoint union of isomorphisms.
Then we can define an 1--form $\epsilon$ on $M\setminus G$ whose value
upon these disjoint components is given by
\[
i_X (\epsilon ) := \mathrm{Tr}_{C/M} (p_C i_C^*(\omega^2)/ \mathrm{Hess} (F)) .
\]
Moreover, we can define commutative and associative product $\circ$ 
by 
\[
\overline{X\circ Y} F = \overline{X}F\cdot \overline{Y} F \ \mathrm{mod} J_F .
\]

Then the scalar product $g: S^2(\cT_M) \to \cO_{M\setminus G}$ defined by
\[
g(X,Y):= i_{X\circ Y}(\epsilon )
\]
is a flat metric, which together with $g$ extends regularly to $M$.
\end{proposition}
 \vspace{3pt}

\subsection{Potentiality and associativity}
\

We describe now the axiomatisation of Saito's frameworks due to B. Dubrovin:
cf.~(\cite{Ma99}, p.19, Definition 1.3 and further on).
 \vspace{3pt}

\begin{definition} A pre--Frobenius manifold is a (super)manifold $M$ endowed with an
affine flat structure $\cT_M^f$ as in~\ref{S:cft} above; with a compatible metric $g$
(i.~e. $g$ is constant upon flat fields); and with an even symmetric tensor 
tensor $A: S^3(\cT_M) \to \cO_M$.
 \vspace{3pt}

This pre--Frobenius manifold is called potential one, if
everywhere locally there exists an even section $\Phi$ of $\cO_M$ such that restriction of $A$
upon $\cT^f_M$ can be written as
\[
A(X,Y,Z)= (XYZ)\Phi .
\]
\end{definition}
 \vspace{3pt}

Upon flat vector fields of such a manifold, we can introduce an even multiplication $\circ$
bilinear over constants such that
\[
A(X,Y,Z) = g(X\circ Y,Z)= g(X, Y\circ Z) .
\]

It is commutative and associative, and can be extended to an $\cO_M$--bilinear,
commutative and associative product $\cT_M\otimes_{\cO_M} \cT_M \to \cT_M$
also denoted $\circ$. 
 \vspace{3pt}

If we choose local flat coordinates $(x^a)$ and respective local basis of tangent
fields $(\partial_a)$, we can write
\[
(\partial_a\circ \partial_b\circ \partial_c)\Phi = \partial_a\partial_b\partial_c \Phi,
\]
and then compatibility of $\Phi$ and $g$ will mean that
\[
\partial_a \circ \partial_b = \sum _c \Phi_{ab}^c \partial_c, \quad \Phi_{ab}^c := \sum_e 
( \partial_a\partial_b\partial_c \Phi ) g^{ec} ,\quad (g^{ab}) := (g_{ab})^{-1}.
\]
Notice that the last formula should be read as an inverted matrix.
 \vspace{3pt}
 
Rewriting the associativity of $\circ$ in the usual way as $( \partial_a\circ \partial_b)\circ \partial_c
= \partial_a\circ (\partial_b\circ \partial_c )$ we obtain a non--linear system of {\it Associativity Equations}, partial
differential equations for $\Phi$:
 \[
\forall a,b,c,d :\quad \sum_{ef} \Phi_{abe}g^{ef}\Phi_{fcd} = (-1)^{a(b+c)} \sum_{ef} \Phi_{bce}g^{ef}\Phi_{fad}.
\]
In the community of physicists, they
are known as {\it WDVV} (Witten--Dijkgraaf--Verlinde--Verlinde) equations.
 \vspace{3pt}

\subsection{Structure connections, flatness, and potentiality} 
\

To conclude this section, we describe below important criteria of potentiality
and associativity expressed in terms of the structure connection of a pre--Frobenius
manifold.
 \vspace{3pt}

First, introduce the connection $\nabla_0 : \cT_M \to \Omega^1_M\otimes_{\cO_M} \cT_M$,
uniquely defined by the horizontality of $\cT^f_M$. Of course, it extends to
the differential upon $ \Omega^*_M\otimes_{\cO_M} \cT_M$ in the standard way.
 \vspace{3pt}

One easily sees that $\nabla_0$ can be further extended to a pencil of connections $\nabla_{\lambda}$
depending on an even parameter $\lambda$: the respective covariant derivative is
\[
\nabla_{\lambda ,X}(Y) := \nabla_{0,X}(Y)+ \lambda X\circ Y.
\]
We will refer to this pencil as the structure connection of our pre--Frobenius manifold.
 \vspace{3pt}

Here is our final result, that can be checked by direct calculations (cf.~\cite{Ma99}, 1.5 and 1.6).
Put $\nabla_{\lambda}^2 =\lambda^2 R_2 + \lambda R_1$. 
 \vspace{3pt}

\begin{proposition}
\ 

\begin{enumerate}
\item Potentiality of $(M,g,A)$ is equivalent to the vanishing of $R_1$,
that in turn is equivalent to the identity holding for all local tangent fields:
\[
\nabla_X(Y\circ Z) -  (-1)^{XY}\nabla_Y (X\circ Z) +X\circ \nabla_YZ - (-1)^{XY} Y\circ \nabla_XZ
- [X,Y]\circ Z  = 0.
\]

\item Associativity of $(M,g,A)$ is equivalent to the vanishing of $R_2$.
\end{enumerate}
\end{proposition}
 \vspace{3pt}
 
 
 \section{Convex cones and families of probabilities}\label{S:covcone}
 \vspace{3pt}

\subsection{Basic example: probability distributions on finite sets}~\label{S:bcfd}
\

Consider a finite set $X$. {\it A probability distribution} $P_X$ on $X$
is a map $P_X: X\to \R, x\mapsto p_x\in [0,1],$ such that $\sum_{x\in X} p_x=1$.
 \vspace{3pt}

The simplest geometric image of the set of all probability distributions on $X$
is the simplex $\Delta_X$ spanned by the end--points of basic coordinate vectors in $\R^X$.
We will also consider its maximal open subset ${}^{\circ}\Delta_X$:
\[
{}^{\circ}\Delta_X := \{(p_x) \,|\, 0< p_x < 1 \ for \ all \ x\in X\}.
\]

The existence of highly non--trivial geometries (in particular, $F$--geometry) naturally supported by such simplices
was one of the first discoveries in the domain of future ``Geometry of Information''.
(As we mentioned, another developments led through homological and homotopical algebra.)
 \vspace{3pt}

The earliest sources here are~\cite{Ch64} and~\cite{Ch65}; see also the monograph~\cite{Ch82} and~\cite{MoCh89},\cite{MoCh91-1},\cite{MoCh91-2}. 
One of the contemporary
expositions is given in~\cite{Mar19}.
 In order to avoid set--theoretical difficulties,
we will be working in a {\it fixed small universe}. 
 \vspace{3pt}

We start with geometry.
 \vspace{3pt}

\subsection{Convex cones, potentiality, and $F$--structures}~\label{S:ccpF}
\

The union of all oriented half--lines in $\R^X$ starting at $(0,\dots ,0)$ and containing
a point of ${}^{\circ}{\Delta}_X$ is a particular case of the general class
of open {\it convex cones}. We use this terminology
here in the sense of~\cite{Vi63}, Ch.1, Introduction, Def.1. 
 \vspace{3pt}

 Namely, let $R$ be a finite dimensional real linear space
(former $\R^X$). By definition, a cone $V\subset R$ is a non--empty subset,
closed with respect to addition and multiplication by positive reals. Moreover, the closure
of $V$ should not contain a real linear subspace of positive dimension.
 \vspace{3pt}

Following~\cite{Vi63}, Ch.1, section 2, we will now introduce the definition and state
main properties of {\it characteristic functions} of general convex cones.
 \vspace{3pt}

Let again $R$ be an oriented finite dimensional real affine space, $R^{\prime}$ its dual space.
We will denote the value of $x^{\prime}\in R^{\prime}$ upon $x\in R$
as $\langle x,x^{\prime} \rangle$.
 \vspace{3pt}
 
\begin{definition} Let $V\subset R$ be a convex cone and $vol_{V^{\prime}}$
be a differential form of maximal degree (a volume form) invariant
with respect to translations in $R^{\prime}$. 
 \vspace{3pt}

The function $\varphi : V\to \R$ defined by
\[
\varphi_{_{V}} (x) := \int_{V^{\prime}} e^{- \langle x,x^{\prime} \rangle} vol_{V^{\prime}}
\]
is called a characteristic function of $V$.
\end{definition}

Since translation invariant volume forms are defined up to a positive constant factor,
the same is true for characteristic functions.
 \vspace{3pt}

Consider now the cone $V$ as a smooth manifold, whose tangent space at any point $x$
can (and will) be canonically identified with $R$, by the parallel transport
identifying $x\in V$ with $0\in R$. Fixing an affine coordinate system $(x^i)$ in $R$,
put
\[
g_{ij}:= \partial^2 \ln \varphi_{_{V}}/\partial x^i \partial x^j \  .
\]
 \vspace{3pt}

The main result from~\cite{Vi63}, needed here, is the following theorem.
 \vspace{3pt}

\begin{theorem}~\label{Th:sqf}
\ 

\begin{enumerate}
 \item The symmetric quadratic form $\sum_{i,j} g_{ij}dx^i dx^j$ determines
a Riemannian metric on $V$.
 \vspace{3pt}

\item The respective metric defines the torsionless canonical connection
on the tangent bundle $\cT_V$ whose components in any affine coordinate system
are
\[
\Gamma^i_{jk} = \frac{1}{2} \sum_l g^{il} \frac{\partial^3 \ln \varphi}{\partial x^j \partial x^k \partial x^l} ,
\]
with
\[
\quad \sum_j g^{ij} g_{jk} =\delta^i_k .
 \]
 \vspace{3pt}

\item Hence the formula
\[
\sum_i a^i \partial_{x_i} \circ \sum_j b^j\partial_{x_j} :=
- \sum_{i,j,k} \Gamma^i_{jk} a^jb^k \partial_{x_i}
\]
defines on $\cT_V$ a commutative $\R$--bilinear composition.
\end{enumerate}
\end{theorem}
 \vspace{3pt} 
 
At this point, the reader should turn back and compare the statement of Theorem~\ref{Th:sqf} constructions involving the Hessian in the Definition~\ref{D:sf}, and subsequent treatment of 
Associativity Equations and potentiality. Clearly, geometry of convex cones
(in particular, cones generated by probability distributions upon finite sets)
provides strong analogies with theory of unfolding spaces of singularities.

\smallskip

In particular, convex cones admit families of $F$--structures depending on the choice
of an affine coordinate system on $R$.
 \vspace{3pt}

\subsection{$\sigma$--algebras and categories of probability distributions}
\

In order to extend the notion of a probability distribution upon 
possibly infinite sets $X$, and to pass to categorical constructions,
 we must recall the definition of a $\sigma$--algebra.
 \vspace{3pt}
 
 Here is a summary of main participants of the game (omitting certain details).
 \vspace{3pt}

Consider a set $X$ and a collection of its subsets $\cF$ satisfying the following restrictions:
 \vspace{3pt}

$X\in \cF$; if $U,V\in \cF$, then $U\setminus V\in \cF$,
so in particular $\emptyset \in \cF$; for any {\it countable} subcollection of
$\cF$, the union of its elements belongs to $\cF$.
 \vspace{3pt}

Such a pair $(X , \cF)$ is called a $\sigma$--{\it algebra.}
 \vspace{3pt}

From the definition it follows that: 
 \vspace{3pt}
 \begin{enumerate}[(a)]
 
\item Intersection of all elements of a countable subcollection of $\cF$ belongs to $\cF$.
 \vspace{3pt}
 
\item If a collection $\cF$ is a {\it countable partition} of $X$, and $\cF^{\prime}$
is the collection formed by all unions of parts of this partition,
then $(X,\cF^{\prime})$ is a $\sigma$--algebra.
\end{enumerate}
 \vspace{3pt}

Given a $\sigma$--algebra $(X ,\cF )$, we will be considering
measures and probability measures/distributions on it.
 \vspace{3pt}

Generally, let $(S, +,0)$ be a commutative semigroup with composition law $+$
and zero element. Then an $S$--valued measure $\mu$ on $(X , \cF )$ 
is a map $\mu :\, \cF \to S$ such that $\mu (\emptyset )=0$,
and $\mu (X\cup Y) + \mu (X\cap Y)= \mu(X) + \mu (Y)$.
 \vspace{3pt}

 Such a measure is called {\it a probability distribution $p$} if $S$ is the additive semigroup
 of non--negative real numbers, and moreover, for any countable subfamily
 $(U_i)$, $i= 1, 2,3, \dots$ of elements of $\cF$ with empty pairwise intersections,
 we have $p(\cup_{i=1}^{\infty} U_i )= \sum_{i=1}^{\infty} p(U_i)$; and
 if such a countable subfamily covers $X$, then the sum of probabilities is 1.
 \vspace{3pt}

\begin{definition}
 Category $CAP$ of probability distributions (\cite{Ch65}) consists of the following data:
 \vspace{3pt}
 
 \begin{enumerate}
\item An object of $CAP$ is the set $Cap (X , \cF )$ of all probability distributions 
on a $\sigma$--algebra $(X , \cF)$.
 \vspace{3pt}
 
\item One (Markov) morphism $\Pi \in Hom_{CAP} (Cap (X_1 , \cF_1 ), Cap (X_2 , \cF_2 ))$
is given by a ``transition measure'', that is a function $\Pi \{* | x^{\prime}\}$ upon $\cF_2\times X_1$
such that for a fixed $U\in \cF_2$, $\Pi \{U | x_1\}$ is $\cF_1$--measurable function
on $X_1$, and for a fixed $x_1 \in X_1$,  
$\Pi \{U | x_1\}$ is a probability distribution upon $\cF_2$.
 \vspace{3pt}
 
Explicitly, such $\Pi$ sends the probability distribution $P_1\in Cap(X_1, \cF_1)$
to the probability distribution $P_2\in Cap(X_2, \cF_2)$ given by
 \[
P_2(X_2 | x_1) := \int_{X_1} \Pi\{ * | x_1\} P_1\{dx_1\}.
\]
\end{enumerate}
\end{definition}

For (more or less evident) description of identical morphisms and composition of morphisms,
see~\cite{Ch65}.
 \vspace{3pt}

Below, we will return from general convex cones to the ones obtained
from $\Delta_X$ by passing to the union of all oriented half--lines in $\R^X$
connecting the origin with a point in ${}^{\circ}\Delta_X$.
Clearly, the boundary of such a cone is a union of cones of the same type
 with vertices 
corresponding to elements of all subsets $\{i_1, \dots , i_m\} = \{1,\dots ,n\}$.
Geodesics of the respective metrics are simply
segments of affine lines in $R$, although the metrics themselves
blow up to infinity near each respective face.
 \vspace{3pt}

This makes it possible to bridge two different paths from the intuitive
image ``description of a global space by approximating it with finite subsets of points'':
 \vspace{3pt}
 
 \begin{enumerate}[1)]
\item Passing from probability distributions on finite subsets to the 
probability distribution on the whole $\sigma$--algebra (\cite{CoGw17},
6.1.2).
 \vspace{3pt}

\item Passing from a simplicial set to the topology of its
geometric realisation (\cite{GeMa03}, I.2, Definition 1, p. 6).
\end{enumerate}
 \vspace{3pt}

In order to enrich simplicial algebra with information geometry,
it is necessary to use the categorical lift of simplicial constructions from the category
of finite sets to a category $CAP$. We hope to return to this challenge later.
 \vspace{3pt}
 

\section{Statistical manifolds and paracomplex structures}\label{S:statman}
\vspace{3pt}

\subsection{Paracomplex geometry}
\

 The algebra of paracomplex numbers (cf.~\cite{CrFoGa96}) is defined as the real vector space $\fC = \R\oplus \R$ with the multiplication
\[
(x,y) \cdot (x',y') = (xx' + yy', xy' + yx'). 
\]
Put $\varepsilon : =(0,1)$.  Then $\varepsilon^2 =1$, and moreover
\[
\fC =\R+\e\R= \{z=x+\e y \, |\, x,y \in \R \}.
\]
Given a paracomplex number $z_{+} = x+\varepsilon y$, its conjugate is defined by $z_{-}:= x-\e y$.
We denote by $\fC^* = \{x+\e y\, |\, x^2 -y^2 \ne 0 \}$ the group of invertible elements of $\fC$.
\vspace{3pt}

Let $E_{2m}$ be a $2m$-dimensional real affine space. {\it A paracomplex structure} on $E_{2m}$ is an endomorphism $\fK: E_{2m} \to E_{2m}$ such that $\fK^2=I$,
and the eigenspaces $E_{2m}^+, E_{2m}^-$ of $\fK$ with eigenvalues $1,-1$ respectively, have the same dimension. 
The pair $(E_{2m},\fK)$ will be called a {\it paracomplex affine space.} 
\vspace{3pt}

Finally, {\it a paracomplex manifold} is a real manifold $M$ endowed with a paracomplex structure $\fK$ that admits an atlas of paraholomorphic coordinates (which are functions with values in the algebra $\fC = \R + \e\R$ defined above), such that the transition functions are paraholomorphic.
\vspace{3pt}

Explicitly, this means the existence of local coordinates $(z_+^\alpha, z_-^\alpha),\, \alpha = 1\dots, m$ such that
paracomplex decomposition of the local tangent fields is of the form
\[
T^{+}M=span \left\{ \frac{\partial}{\partial z_{+}^{\alpha}},\, \alpha =1,...,m\right\} ,
\]
\[
T^{-}M=span \left\{\frac{\partial}{\partial z_{-}^{\alpha}}\, ,\, \alpha =1,...,m\right\} .
\]
Such coordinates are called {\it adapted coordinates} for the paracomplex structure $\fK$.
\vspace{3pt}

If $E_{2m}$ is already endowed with a paracomplex structure $\fK$ as above,
we define {\it the paracomplexification of $E_{2m}$} as $E_{2m}^\fC = E_{2m} \otimes_{\R} \fC$ and we extend $\fK$ to a $\fC$-linear endomorphism $\fK$ of $E_{2m}^\fC$. Then, by setting
\[
E_{2m}^{1,0} = \{v\in V^\fC \, |\, \fK v=\e v\}=\{v+\e\fK v\, |\, v \in E_{2m}\},
\]
\[
E_{2m}^{0,1} = \{v\in V^\fC \, |\, \fK v= -\e v\}=\{v-\e\fK v\, |\, v\in E_{2m}\},
\]
we obtain $E_{2m}^\fC =E_{2m}^{1,0}\oplus E_{2m}^{0,1}$.
\vspace{3pt}

 We associate with any adapted coordinate system $(z_{+}^{\alpha}, z_{-}^{\alpha})$ a paraholomorphic coordinate system $z^{\alpha}$ by 
\[
z^\alpha\, =\, \frac{z_{+}^{\alpha}+z_{-}^{\alpha}}{2} +\e\frac{z_{+}^{\alpha}-z_{-}^{\alpha}}{2}, \alpha=1,...,m .
\]
\vspace{3pt}
 
We define the paracomplex tangent bundle as the $\R$-tensor product $T^\fC M = TM \otimes \fC$ and we extend the endomorphism $\fK$ to a $\fC$-linear endomorphism of $T^\fC M$. For any $p \in M$, we have the following decomposition of $T_{p}^\fC M$:
\[
T_p^\fC M=T_p^{1,0}M \oplus T_p^{0,1}M\,
\]
where 
\[
T_p^{1,0}M = \{v\in T_p^\fC M | \fK v=\e v\}=\{v+\e \fK v| v \in E_{2m}\} ,
\] 
\[
T_p^{0,1}M = \{v\in T_p^\fC M | \fK v= -\e v\}=\{v-\e \fK v|v\in E_{2m}\}
\]
are the eigenspaces of $\mathfrak{K}$ with eigenvalues $\pm \e$.
The following paracomplex vectors 
\[
\frac{\partial}{\partial z_{+}^{\alpha}}=\frac{1}{2}\left(\frac{\partial}{\partial x^{\alpha}} + \e\frac{\partial}{\partial y^{\alpha}}\right),\quad \frac{\partial}{\partial{z}_{-}^{\alpha}}=\frac{1}{2}\left(\frac{\partial}{\partial x^{\alpha}} - \e\frac{\partial}{\partial y^{\alpha}}\right)
\]
form a basis of the spaces $T_p^{1,0}M$ and $T_p^{0,1}M$.
\vspace{3pt}
 
Useful constructions from the theory of paracomplex differential forms are collected in~\cite{AlMeTo09}, as well as~\cite{CoMaSa04}, \cite{CoMaMoSa05} , \cite{Lib52}).
 In particular, one can define the Dolbeault paracomplex (see~\cite{CoMaSa04} for details).
\vspace{3pt}

\subsection{Convex cones and paracomplex geometry}
\

Before applying this machinery to the spaces of probability distributions on finite sets
(cf. Sec.~\ref{S:bcfd}), we should explain why we cannot extend it to the more general setting
of (subspaces of) finite--dimensional convex cones. 
\vspace{3pt}

The main reason is this: in order to establish the connection with $F$--manifolds,
we need to have a paracomplex analogue of Theorem~\ref{Th:sqf} in which real 
differential forms and Riemannian metrics would be replaced by their
paracomplex versions. But it turns out, that this is possible only
for a narrow subclass of convex cones that unmistakably 
singles out probability distributions on finite sets. 
\vspace{3pt}

This subclass is the last one in the Vinberg's list of such cones
that are irreducible ones with respect to direct sums (cf. \cite{Vi60}, \cite{Vi63}, \cite{Vi65}).
\vspace{3pt}

\begin{proposition}~\label{P:Vclass}
Each irreducible homogeneous self--dual cone belongs to one
of the following classes:
\vspace{3pt}

\begin{enumerate}
\item The cone $M_+(n, \R)$ of $n \times n$ real positive matrices.
\vspace{3pt}

\item The cone $M_{+}(n, \C)$ of $n \times n$ complex positive matrices.
\vspace{3pt}

\item The cone $M_{+}(n, \H)$ of $n \times n$ quaternionic positive matrices.
\vspace{3pt}

\item The cone $M_{+}(3, \O)$ of $3 \times 3$ positive matrices whose elements are in $\O$, the Cayley algebra (also known as the Octonionic algebra).
\vspace{3pt}

\item The cone $M_{+}(n, \fC)$ of $n \times n$ paracomplex positive matrices. 
\end{enumerate}
\end{proposition}
\vspace{3pt}
Recall that a matrix is positive if it is self-adjoint and its eigenvalues are positive.

\vspace{3pt}

\begin{definition} The structure of Jordan algebras~\cite{JoNeWi34} on a real linear space $\fM$
is determined by two polylinear operations:
\vspace{3pt}

\begin{enumerate}
\item binary multiplication $(a,b) \to a\cdot b$,
\vspace{3pt}

\item ternary multiplication $(a,b,c)\to a(bc)$,
satisfying the compatibility axiom
\[
 a\cdot ((a\cdot a)\cdot b) = (a\cdot a)\cdot (a\cdot b) .
\]

Such an algebra is called formally real if from $ \sum_{i=1}^n a_i\cdot a_i =0$ it follows
that all $a_i=0$.
\end{enumerate}
\end{definition}
\vspace{3pt}

\begin{theorem}~\label{Th:iffea}
\ 

\begin{enumerate}
\item The list of algebras in Proposition~\ref{P:Vclass} coincides with the list
of all irreducible finite dimensional formally real algebras.
\vspace{3pt}

\item The irreducible homogeneous self-dual cone associated with such an algebra
$\fM$ is the set of positive elements of a Jordan algebra, i.e. elements represented by positive matrices.
\end{enumerate}
\end{theorem}

\begin{proof} For the proof of this theorem, we refer directly to ~\cite{Vi60} and \cite{Vi65}.
\end{proof}
\vspace{3pt}

From the works studying affine spaces over an algebra of finite rank~\cite{Ca27}, \cite{No63}, \cite{Ro49, Ro97}, \cite{Sh02}, we have the following statement:
 \begin{proposition}\label{P:rank2}
Consider an affine, symmetric space over a Jordan algebra. There exists exactly two affine and flat connections on this space if and only if the algebra is of rank 2, and generated by $\{1, \e\}$ with $\e^2= 1$ or $-1$.
\end{proposition}

In the case where $\e^{2}=-1$, we have a complex structure. Similarly, if $\e^{2}=1$ we have a paracomplex structure .

\begin{proof}
1) Suppose that we are working on an affine space over a Jordan algebra $A$ of rank two, with basis elements $\{e_{1},e_{2}\}=\{1,\epsilon\}$. The affine representation of the algebra, or free module $AE_{n}$, admits a real interpretation in the affine space $E_{2n}$(\ \cite{Ro97}, section 2.1.2). In this interpretation each vector $\bx = (x^{i})\in AE_{n}$ with coordinates $x^{i}=x^{(i,\alpha)} e_{\alpha}$, is interpreted as the vector $\bx=(x^{(i,\alpha)}) \in E_{2m}$. 

\smallskip 

Let us introduce a parametrizable curve $ x^{i} = x^{i}(t) \in E_{2n}$, and a tangent vector $\bw$ to it, at a given point. Our aim is to proceed to the parallel transport of this vector, along that curve in $E_{2n}$. We have a parallel transport of $\bw$ along of the curve $x^{i}= x^{i}(t)$ given by:

\[d\bw+ \Gamma \bw d\bx=0,\] and $\Gamma$ is an affine connection for the $A$-space.
Because of the splitting property, we can write the parallel transport equation in the following way:

\[d(\bw^{(1)}\oplus \bw^{(2)}) +(\Gamma^{(1)}\oplus\Gamma^{(2)})(\bw^{(1)}\oplus \bw^{(2)})d(\bx^{(1)}\oplus \bx^{(2)})=0,\] 
therefore, giving us:
\[
d\bw^{(\alpha)}+ \Gamma^{(\alpha)}\bw^{(\alpha)}d\bx^{(\alpha)} = 0 ,\quad \alpha\in \{1,2\}.
\]
Therefore, we can define an affine connection in $E_{2n}$ having two components with respect to the local coordinates $x^{(i,\alpha)}$.

\smallskip
2) Consider an affine and symmetric space over a Jordan algebra $A$, and suppose that there are 2 flat, affine connections on this space. 
These flat affine connections are constructed from a field of objects, having components: 
\[\Gamma_{jk}^{i}=\Gamma_{jk}^{i \alpha}e_{\alpha} \in A.\]

\smallskip
Suppose that $\bv^{i}=\bv^{(i,\alpha)}e_{\alpha}$ are quantities from the algebra corresponding to a tangent vector $\bv$. 
Then, from the following condition
\[d\bv^{i} + \Gamma^{i}_{j k}\bv^{j} d\bx^{k}= 0,\]
we can define an affine connection in the affine space $E^{2n}$ having the following components
\[\Gamma^{(i,\alpha)}_{(j,\beta)(k,\gamma)}=\Gamma_{j k}^{i s}C_{s \beta}^\delta C^{\alpha}_{\delta \gamma} ,\]
where the $C_{ \beta \gamma}^{\alpha}$ are structure constants of algebra $A$, with respect to the local adapted coordinates $x^{(\alpha,i)}$. 
Now, these objects are indexed by the number of generators of the algebra $A$. Since there exist 2 connections, it impies that $s\in\{1,2\}$ and so that the number of generators of the algebra $A$ is 2. 
\end{proof}

\medskip

\subsection{Projective space and paracomplex structure}\label{S:4.3}
Let $X_{d}$ be a $d$-dimensional surface of the $n$-dimensional (real or complex) projective space $\cP^{n}$ with $d\leq n$.
\begin{definition}
The surface $X_{d}$ is said to be normalized if, at each point $p\in X_{d}$, are associated the two following hyperplanes:
\begin{enumerate}
\item Normal of first type, $P_{I}$, of dimension $n-d$, and intersecting the tangent $d$-plane $T_{p}X_{d}$ at a unique point $p \in X_{d}$. 
\item Normal of second type, $P_{II}$, of dimension $d-1$, and included in the $d$-plane $T_{p}X_d$, not meeting the point $p$.
\end{enumerate}
\end{definition}

This decomposition expresses the duality of projective space. In particular, in the limit case, where $d=n$, then $P_{I}$ is reduced to the point $p$ and $P_{II}$ is the $(n-1)$-surface which does not contain the point $p$. 
This property is nothing but the usual duality of projective space. Note that in this case, $X_{n}$ can be identified with the projective space $\cP^{n}$.

\begin{definition}\label{D:mpairs}
A pair consisting of an $m$-plane and an $(n-m-1)$-plane is called an $m$-pair. 
\end{definition}
\begin{remark} The $0$-pair can be identified with the projective space $\cP^{n}$.\end{remark}
\vspace{5pt}

From~\cite{No47,Sh87}, for normalized surfaces associated to an $m$-pair space, the following properties holds:
\medskip
\begin{lemma}\label{L:pairs}
\ 

\begin{enumerate}
\item The space of $m$-pairs is a projective, differentiable manifold.
\item For any integer $m\geq 0$, a manifold of $m$-pairs contains 2 flat, affine and symmetric connections. 
\end{enumerate}
\end{lemma}

In particular, this leads to the the following proposition:
\begin{proposition}\label{P:isome}
The space of $0$-pairs in the projective space $\cP^n$ is isometric to the hermitian projective space over the algebra of paracomplex number.
\end{proposition}
\begin{proof} see e.g.~\cite{Ro97} section 4.4.5.\end{proof}
\medskip
\begin{proposition}\label{P:zero}
Suppose that $(X,\mathcal{F})$ is a finite measurable set where the dimension of $X$ is $n+1$, and measures vanish only on an ideal $\mathcal{I}$.
Let $\mathcal{H}_{n}$ be the space of probability distributions on $(X,\mathcal{F})$.
Then, the space $\mathcal{H}_{n}$ is a manifold of 0-pairs.
 \end{proposition}
 
 \begin{proof}
 The $n$-dimensional surface $\mathcal{H}_{n}$ is the intersection of the hyperplane $\mu(X)=1$ with the cone $\cC_{n+1}$ of strictly positive measures, in the affine space $\cW_{n+1}$ of signed bounded measures. It is interpreted as a $n$-dimensional surface (also denoted by $\mathcal{H}_{n}$) of the projective space $\cP^{n}$. Then, the geometrical structure of this surface is inherited from projective geometry. Using the remark in the first paragraph of section 0.4.3 in \cite{Ro97} and the definition\, \ref{D:mpairs} of 0-pairs, one deduces that it corresponds to a manifold of 0-pairs.
 \end{proof}
 
 \medskip
 
\begin{theorem}\label{Th:main}
Suppose that $(X,\mathcal{F})$ is a finite measurable set where the dimension of $X$ is $n+1$, and measures vanish only on an ideal $\mathcal{I}$.
The space $\mathcal{H}_{n}$ of probability distributions on $(X,\mathcal{F})$ is isomorphic to the hermitian projective space over the cone $M_{+}(2,\fC)$.
\end{theorem}
\begin{proof}
This is a consequence of applying Proposition\, \ref{P:zero}, Lemma\,\ref{L:pairs} and finally Proposition\, \ref{P:rank2}.
\end{proof}

\subsection{Paracomplex potentiality of spaces of probability distributions.}
\vspace{3pt}

Now we will describe explicitly the analogues of local potentials $\varphi$ from Sec.~\ref{S:covcone}
in the paracomplex geometry. Using this description, we will state the paracomplex version of Theorem~\ref{Th:sqf} 
for cones of probability
distributions.
\vspace{3pt}

\begin{theorem} [Paracomplex Dolbeault lemma]~\label{Th:paraD}
 Any (local) potential $\varphi$ on a cone of probability distributions determine the
local paracomplex Dolbeault $(1,1)$--form 
\[
\widetilde{\omega} := \partial_{+} \partial_{-} \varphi = \varepsilon \partial \overline{\partial} \varphi.
\]
The potential $\varphi$ is defined uniquely modulo subspace of local functions
$Ker\, \partial_{+} \partial_{-}$.
\end{theorem}
\vspace{3pt}

\begin{proof} The proof uses an explicit construction of the paracomplex structure encoded 
in the direct sum $R\oplus R^{\prime}$ from Sec.~\ref{S:ccpF} above.
\vspace{3pt}

Let $(V,I,g)$ be a para--K\"ahler manifold with para--K\"ahler form $\omega$. Consider
a point $p$ on $V$, and an an open neighborhood $U$ of $p$.
Let $(z^{i}_{\pm})$ be adapted coordinates defined on $U$ and mapping $U$ onto the product of two simply connected open sets $U^{\pm}\subset \R^n$, where $n=dim_{\fC}V$. Moreover, assume that $z^{i}_{\pm}(p)=0$. 
\vspace{3pt}

Suppose that $\partial_{+}\theta^{+}=0$ on $U\cong U^{+}\times U^{-}$. The Dolbeault paracomplex technique shows that there exists a function $\varphi^+$ on $U$, 
given by: 
\[
 \varphi^+:= \int_{(0,z_{-})}^{(z_{+},z_{-})}\theta^{+}.
\]

The integration is over any path from $(0,z_{-})$ to $(z_{+},z_{-})$ contained in $U^{+}\times \{ z_{-}\}.$ 

From the condition that $\partial_{+}\theta^{+}=0$, it follows that the integral is path independent and that the one-form $\theta^+$ restricted to $U^{+}\times \{ z_{-}\}$ is closed (and thus exact), since $U^+$ is simply connected. 
\vspace{3pt}

We now show that there exists a real valued function $\varphi$ defined in some 
simply connected open neighbourhood $U$ of $p$ such that 
$\omega=\partial_{-}\partial_{+}\varphi$ on $U$. The function $\varphi$ is unique up to addition of a real--valued function $f$ satisfying the equation $\partial_{-}\partial_{+}f=0$. Any such function is of the form $f=f_{+}+f_{-}$, where $f_{\pm}: U\to \R$ satisfying the equation $\partial_{\mp}\partial_{\pm}f=0$.

The first cohomology of $U$ vanishes, so $H^{1}(U, \R)=0.$ Since $\omega$ is closed, there exists a one--form 
$\theta$ such that $\omega=d\theta$. We decompose $\theta$ into its homogeneous components: 
$\theta=\theta^++\theta^{-}$, $\theta^+\in \Omega^{1,0}(U),$ $\theta^{-}\in \Omega^{0,1}(U)$. Then
\[
d\theta= \partial_{+}\theta^++(\partial_{-}\theta^{+}+ \partial_{+}\theta^{-}) + \partial_{-}\theta{-}.
\]
From the fact that $\omega$ is of the type (1,1), we obtain the equations: 
\[
\partial_{\pm}\theta^{\pm}=0,\quad \text{and} \quad \partial_{-}\theta^{+}+\partial_{+}\theta^{-}=\omega.
\]

Therefore, there exist two real-valued functions $\varphi^{\pm}$ such that $\partial_{\pm}\varphi^{\pm} = \theta^{\pm}$. Assuming that 
$\varphi:=\varphi^{+}-\varphi^{-}$, we have:
\[
\partial_{-}\partial_{+}\varphi= \partial_{-}\partial_{+}\varphi^{+} + \partial_{+}\partial_{-}\varphi^{-}= \partial_{-}\theta^{+}+ \partial_{+}\theta^{-} =\omega.
\]
It is clear that the function $\varphi$ is unique up to adding a solution of $\partial_{-}\partial_{+}f= 0$ (in fact, any solution is of the form $f=f_{+}+f_{-}$, where $\partial_{\pm}f_{\pm}=0.$).
Let us consider $\partial_{+}f=\sum f_{i}^{+}dz_{+}^{i}, \quad $ with $\quad f_{i}^{+}=\frac{\partial f}{\partial z^i_{+}},$ 
We get 
\[
0=\partial_{-}\partial_{+}f=\sum \frac{\partial f_{i}^{+}}{\partial z^j_{-}}dz^j_{-}\wedge dz^{i}_{+}.
\]
Therefore, $\frac{\partial f_{i}^{+}}{\partial z^j_{-}}=0$ and the functions $\partial f_{i}^{+}$ depend only on the positive coordinates $z_+$.
So, we obtain 
\[
f=\sum \int_{0}^{z_{+}}f_i^{+}(\xi)d\xi^i+ f_{-}(0,z_{-}), 
\] 
where $\xi=(\xi^1,\dots,\xi^n)$. The path integral is well defined, $U^+$ being simply connected. By a change of notation, we have that $f_{+}=\sum\int_{0}^{z_{+}}f_{i}^{+}(\xi)d\xi^i$, and thus $f=f_{+}+ f_{-}$. 

Conversely, let $\varphi$ be a real-valued function on $U\subset V$ such that $\omega=\partial_{-}\partial_{+}\varphi$ is a non-degenerate two--form. This two--form is closed and of type (1,1). This is equivalent to $I^{*}\omega=-\omega$, which implies that $g:=\omega(I\cdot,\cdot)$ is symmetric i.e. $g(X,Y)=g(Y,X)$ . 
\end{proof}
\vspace{3pt}

\subsection{Projective geometry of statistical manifolds}
\

Returning to the proof of Theorem~\ref{Th:main} above, we recall that the space of probability
distributions over a finite set is endowed with two flat subspaces with flat connections. 
 This property is common to a large class of probability distributions, generalizing the distributions on a finite set. 
\vspace{3pt}
 
 More precisely, let us consider the positive cone $\cC$ of strictly positive measures on a space 
 $(X,\cF)$, vanishing only on an ideal $\cI$ of the $\sigma$--algebra $\cF$ of the $n$--dimensional real space 
 $\cW$ of the signed measures of bounded variations (i.e. signed measures whose total variation $ \Vert \mu \Vert =|\mu |(X) $ is bounded, vanishing only on an ideal $\mathcal{I}$ of the $\sigma$-algebra $\mathcal{F}$).
\vspace{3pt}

Let $\cH\subset \cC$ be the subset of probability distributions defined by the following constraint on measures $\mu\in \cW$:
\[
 \langle 1,\mu\rangle =1 , \text{ where }
\langle f,\mu\rangle= \int_{X}fd\mu.
\]

We associate to any parallel transport $h$ in the covector space $\cW^{*}$ of the space $\cW$ of $\sigma$--finite measures
$f\xrightarrow{h} f+h$,
 an automorphism of the cone $\cC$
\[
\mu \xrightarrow{h} \nu,\, \text{ where }\, \frac{d\nu}{d\mu}\, =\, \exp (h), 
\]
where $d\nu/d\mu$ is the Radon--Nikodym derivative of the measure $\nu$ w.r.t. the measure $\mu$.
This automorphism is a {\it non--degenerate} linear map of $\cW$ which leaves the cone invariant. 
\vspace{3pt}

Let $\fG$ be the group of all automorphisms $h$ such that $h\, =\, \ln\, \frac{d\nu}{d\mu}.$ 
The commutative subgroup of all ``translations'' of the cone $\cC$ is a simply transitive Lie group, so the cone is homogeneous. To this group $\fG$ the associated Lie algebra $\fg$ defines the derivation of the cone. 
\vspace{3pt}

The cone $\cC$ is not invariant w.r.t. the group $\fG$, but, since $\cC\cap h(\cC)\ne \emptyset$ for any $h\in \fG$,
 $\fG$ is 
the so called pseudo--group of automorphisms of $\cC$. 
The subset $\cH$ of probability distributions is a hypersurface in $\cC$ which can be equipped with an paracomplex algebraic structure (see Proposition\, \ref{P:isome} for information about the algebraic structure)

\begin{lemma}\label{L:tor}
The manifold of probability distributions $\cH$ is torsionless.
\end{lemma}
\begin{proof}
We consider the $n$-dimensional affine space over an algebra $A$. By the previous results, we can assume that this algebra is of rank 2. Recall that $A$ is finite-dimensional, unitary, associative. We interpret this as the affine space $E_{2n}$. One particularity is that we have a representation of the algebra such that 
to any generator of $A$ corresponds a unique endomorphism $E_{2n}$ (the structural endomorphisms). 

\smallskip  

We turn our considerations to so-called dyadics, i.e. endomorphisms depending on the constant structures of the algebra. 
Let $\bv$ be a vector in $E_{2n}$, given by $\bv=v^{(\alpha,i)}\be_{(\alpha ,i)}\in E_{2n}$. There corresponds the element $\bV=V^{\alpha}\bE_{\alpha}\in M_{n}(A)$, where $M_{n}(A)$ is the free unitary $A$-module with basis $\bE_{\alpha}$.

\smallskip 

Now, we consider $M_{2n}$ the differentiable manifold defined by $A$. This is given by the space of affine connections. The regular structure defined by the algebra $A$ arises on it in the case where on $M_{2n}$ we have a set of $2$ dyadic tensors, with matrices simultaneously reduced to the form 
\smallskip 
\[\begin{pmatrix}
\hat{C}_k & \cdots&0\\
\vdots &\ddots&\vdots \\
0 & \cdots& \hat{C}_k
\end{pmatrix}\]
where $\hat{C}_k=(C_{jk}^i)$ is defined in the adapted basis $\be_{(\alpha ,i)}$. 
Each tangent vector space serves as a real model (i.e. a representative in the affine space $E_{2n}$) of the module $M_n(A).$

\smallskip 

On the manifolds $M_{2n}$ defined by $A$, we have a field of objects $\Gamma^{\alpha}_{\beta\gamma}= \Gamma^{(\alpha,s)}_{\beta\gamma}\be_s\in A$ in local coordinates.
Since, we have the relation \[\Gamma^{(\alpha,i)}_{(\beta,j)(\gamma,k)}=\Gamma^{\alpha,s}_{\beta\gamma}C^{m}_{sj}C^{i}_{mk}\]
with respect to the local adapted coordinates $x^{(\alpha,i)}$ (see\, \cite{Sh02}, equation (3)).

\smallskip 
From the commutativity relation of the constant structures defining $A$, we have the $C_{ij}^k=C_{ji}^{k}$. Therefore, we have $\Gamma^{\alpha}_{\beta\gamma}=\Gamma^{\alpha}_{\gamma\beta}$, which implies that $\cH$ is torsionless.

\end{proof}

\medskip
\begin{theorem}
The manifold of probability distributions $\cH$ is an $F$--manifold.
\end{theorem}

\begin{proof}
From Lemma\, \ref{L:tor}, the manifold $\cH$ is torsionless and from theorem\, \ref{Th:main}, we know that $\cH$ has a paracomplex structure. 
\vspace{3pt} 

We have shown in Proposition\, \ref{P:zero} and Theorem\, \ref{Th:main} that the manifold of probability distributions is, geometrically speaking, a projective euclidean manifold and that it has a Clifford algebra structure (for further information see as well \cite{MoCh91-1}~\cite{MoCh91-2}~\cite{Ch82}).
 
\vspace{3pt}
 
From section \ref{S:4.3}  it follows that $\cH$ contains two distinct real projective flat subspaces  (see also the Rozenfeld--Yaglom theorem \cite{RoYa51}, p.112). On the other hand, we already know the potentiality of $\cH$. We can now define respective closed 2--paracomplex form 
$\tilde{\omega}$. 
\vspace{3pt}

From the proof of the paracomplex Dolbeault Lemma, it follows that locally there exists a real-valued function 
$\varphi$ (potential) such that 
\[
\tilde{\omega}=\partial_{+}\partial_{-}\varphi=\e \partial \overline{\partial} \varphi.
\]
The potential $\varphi$ is defined up to addition of a function $f$ satisfying the condition $\partial_{+}\partial_{-}f\, =\, 0.$
\vspace{3pt}

Going back to the multiplication operation $\circ$, we see again that for any pair of flat vector fields, $X, Y$, there exists a vector field $C$ (a potential vector field) such that the multiplication operation is given by $X \circ Y=[X,[Y,C]]$. Therefore $\cH$ is an $F$--manifold.
\end{proof}


\begin{thebibliography}{9}



\bibitem[AlMeTo09]{AlMeTo09} D. V. Alekseevsky, C. Medori, A. Tomassini, 'Homogeneous para-K\"ahler Einstein manifolds',
{\it Uspekhi Mat. Nauk} {\bf 64:1} (2009) 3--50.


\bibitem[Am85]{Am85} S.--I. Amari, {\it Differential Geometrical methods in Statistics}, Lecture Notes in Statistics {\bf 28} (Springer Verlag, 1985).

\bibitem[AmNa00]{AmNa00} S.--I. Amari, H. Nagaoka, {\it Methods of Information Geometry,} Transl. Math. Monographs, vol. {\bf 191} (AMS Providence, 2000).


\bibitem[BaKo72]{BaKo72} S. Barannikov, M. Kontsevich, `Frobenius manifolds and formality of Lie algebras of polyvector fields',
Preprint, 1997, arXiv:9710032.

\bibitem[Beh97]{Beh97} K. Behrend, `Gromov--Witten invariants in algebraic geometry', {\it Inv. Math.} {\bf 127} (1997) 601--617.

\bibitem[BehMa96]{BehMa96} K. Behrend, Yu. Manin, `Stacks of stable maps and Gromov--Witten invariants',
{\it Duke Math. J.} {\bf 85:1} (1996) 1--60.

\bibitem[BeIo78]{BeIo78} J. Belissard, B. Iochum, `Homogeneous self--dual cones versus Jordan algebras.
The theory revisited', {\it Ann. Inst. Fourier} {\bf 28:1} (1978) 27--67.


\bibitem[BuCoNe00]{BuCoNe00} G. Burdet, Ph. Combe, H. Nencka,
'Statistical manifolds: the natural affine-metric structure of probability theory.' {\it  Mathematical physics and stochastic analysis} (Lisbon, 1998), 138--164, World Sci. Publ., River Edge, NJ, 2000. 

\bibitem[Ca27]{Ca27} E. Cartan, `La g\'eom\'etrie des groupes de transformation,' {\it Journ. de math\'ematiques 
pures et applique\'ees 9-e s\'erie} {\bf 6} (1927) 1--120.

 
\bibitem[Ch64]{Ch64} N. N. Chentsov, `Geometry of the manifold of probability distributions',
{\it Dokl. Akad. Nauk SSSR} {\bf 158:3} (1964) 543--546.

\bibitem[Ch65]{Ch65} N. N. Chentsov, `Categories of mathematical statistics', {\it Dokl. Akad. Nauk SSSR} {\bf 164:3} (1965) 511--514.

\bibitem[Ch82]{Ch82} N. N. Chentsov, {\it Statistical Decision Rules and Optimal Inference,} Transl. Math. Monographs, vol. {\bf 53} (AMS Providence, 1982).

\bibitem[CoGw17]{CoGw17} K. Costello, O. Gwilliam. {\it Factorization Algebras in Quantum Field Theory, Vol. 1}, New Math. Monographs
(Cambridge University Press, 2017).


\bibitem[CoMaSa04]{CoMaSa04} V. Cortes, C. Mayer, F. Saueressig, `Special geometry of Euclidean supersymmetry I. Vector multiplets',
{\it J. High Energy Phys} {\bf 028} (2004) 73.

\bibitem[CoMaMoSa05]{CoMaMoSa05} V. Cortes, C. Mayer, T. Mohaupt, F. Saueressig, `Special geometry of Euclidean supersymmetry II. Hypermultiplets 
and the $c$--map', {\it J. High Energy Phys.} {\bf 025} (2005) 27.

\bibitem[CrFoGa96]{CrFoGa96} V. Cruceanu, P. Fortuny, P.M. Gadea, `A survey on paracomplex geometry,
{\it Rocky Mountain Journal of Mathematics} {\bf 26:1} (1996). 

\bibitem[D96]{D96} B. Dubrovin, `Geometry of 2D topological field theories', {\it Integrable Systems and Quantum Groups}
{\bf 1620} (1993) 120--348.


\bibitem[GeMa03]{GeMa03} S. I. Gelfand, Yu. I. Manin, {\it Methods of Homological Algebra,} 2nd edn, 
Springer Monographs in Mathematics (Springer Verlag, 2003).

\bibitem[HeMa99]{HeMa99} C. Hertling, Yu. I. Manin, `Weak Frobenius manifolds', {\it Int. Math. Res. Notices}
{\bf 6} (1999) 277--286.

\bibitem[HeVa18]{HeVa18} C. Hertling, A. Varchenko, `Potentials of a Frobenius--like structure', {\it Glasgow Math. J.} {\bf 60} (2018)
681--693.

\bibitem[JoNeWi34]{JoNeWi34}
P. Jordan, J. von Neumann, E. Wigner, `On an algebraic generalization of the quantum formalism', 
{\it Ann. of Math.} {\bf 36} (1934) 29--64.

\bibitem[K\"o58]{K\"o58}
N. K\"ocher, `Die Geod\"atischen von Positivit\"atsbereichen', {\it Math. Ann.} {\bf 135} (1958) 192--202.

\bibitem[KoMa94]{KoMa94} M. Kontsevich, Yu. I. Manin, `Gromov--Witten classes, quantum cohomology, and enumerative
geometry', {\it Comm. Math. Phys.} {\bf 164:3} (1994) 525--562.

\bibitem[LiZu93]{LiZu93} B. H. Lian, G. Zuckerman, `New perspectives on the BRST--algebraic structure of string theory',
{\it Comm. Math. Phys.} {\bf 154} (1993) 613--646.

 \bibitem[Lib52]{Lib52}
 P. Libermann `Sur les structures presque paracomplexes', {\it C. R. Acad. Sci. Paris} {\bf 234} (1952) 2517--2519. 
 
\bibitem[Ma98]{Ma98} Yu. I. Manin, `Three constructions of Frobenius manifolds: a comparative study', 
{\it Asian J. Math.} {\bf 3:1} (1999) 179--220 (Atiyah's Festschrift), arXiv:9801006.

\bibitem[Ma99]{Ma99} Yu. I. Manin, { \it Frobenius Manifolds, Quantum Cohomology, and Moduli Spaces}, (AMS Colloquium
Publications, Vol. 47, 1999).


\bibitem[Ma19]{Ma19} Yu. I. Manin, `Mirrors, Functoriality, and Derived Geometry', {\it In: Handbook for Mirror Symmetry
of Calabi--Yau and Fano Manifolds,} Higher Education Press, Beijing, and International Press.
Advaned Lectures in Mat., {\bf 47} (2019) 229--251. arXiv:1708.02849.

\bibitem[Mar19]{Mar19} M. Marcolli, `Gamma spaces and information', {\it Journ. of Geometry and Physics} {\bf 140} (2019) 26--55.

\bibitem[MaMar20]{MaMar20} Yu. I. Manin, M. Marcolli, `Homotopy theoretic and categorical models of neural
information theory', {\it in preparation}.

\bibitem[MoCh89]{MoCh89} E. A. Morozova, N. N. Chentsov, `Markov invariant geometry on state manifolds', {\it J. Soviet Math.}
{\bf 56:5} (1991, Russian original 1989) 2648--2669.

\bibitem[MoCh91-1]{MoCh91-1}
 E. A. Morozova, N. N. Chentsov, `Projective Euclidean geometry and noncommutative probability theory', 
 {\it Trudy Mat. Inst. Steklov} {\bf 196} (1991) 105--113.
 Proc. Steklov Inst. Math., {\bf 196} (1992), 117--127 (English).	
 
 \bibitem[MoCh91-2]{MoCh91-2}
 E. A. Morozova, N. N. Chentsov, `Natural geometry of families of probability laws', {\it Itogi Nauki i Tekhniki., Ser. Sovrem. Probl. Mat. , Fund. Napr.} {\bf 83}, (1991) 133--265.
 
 \bibitem[No47]{No47} A.P. Norden, `La connexion affine sur les surfaces de l'espace projectif', {\it Rec. Math. [Mat. Sbornik] N.S.} {\bf 20(62)} (1947) 263--281.
 
\bibitem[No63]{No63} A.P. Norden, `Cartesian composition spaces', {\it Izv. Vyssh. Uchebn. Zaved. Matematika} {\bf 35:4} (1963) 117--128.

\bibitem[Od85]{Od85} T. Oda, `K. Saito's period map for holomorphic functions with isolated critical points',
{\it Adv. Studies in Pure Math.} {\bf 10} (1987), Algebraic geometry, Sendai (1985) 591--648.


\bibitem[Ro49]{Ro49} B. A. Rozenfeld, `The projective differential geometry of the family of pairs $P^{m}+P^{n-m-1}$ in $P^{n}$', 
{\it Mat. Sbornik N. S.} {\bf 24} (66)(1949) 405--428.

\bibitem[RoYa51]{RoYa51} B.A. Rozenfeld, I. M. Yaglom, `On the geometries of the simplest algebras', 
{\it Mat. Sbornik. N. S. } {\bf 28} (70) (1951) 205--216.

\bibitem[Ro97]{Ro97} B. A. Rozenfeld, `Geometry of Lie groups', 
{\it Mathematics and its application} {\bf 393} Springer-Sciences+Business Media.B.V.(1997).



\bibitem[Sa82]{Sa82} K. Saito, `Primitive forms for a universal unfolding of a function with an isolated critical point',
{\it Journ. Fac. Sci. Univ. Tokyo} Sec. IA, {\bf 28:3} (1982) 775--792.

\bibitem[Sa83]{Sa83} K. Saito, `Period mapping associated to a primitive form', {\it Publ. Res. Inst. Math. Sci. Kyoto
Univ.} {\bf 19} (1983) 1231--1264.

 \bibitem[Sh87]{Sh87} A. P. Shirokov, `Affine connection spaces (some aspects of the method of normalization of A.P. Norden's),'{\it Jornal of Soviet Mathematics } {\bf 37}, (1987), 1239--1253.
 
 \bibitem[Sh02]{Sh02} A. P. Shirokov, `Spaces over algebras and their applications', {\it Journ. Math, Sci.} {\bf 108} (2002) 232--248. 

\bibitem[Vi60]{Vi60}
E. B. Vinberg, `Homogeneous cones', {\it Doklady Acad. Nauk. USSR} , {\bf 133} (1960), 9--12.

\bibitem[Vi63]{Vi63} E. B. Vinberg, `The theory of homogeneous convex cones', {\it Trans. Mosc. Math. Soc.}
 {\bf 12} (1963) 340--403.

\bibitem[Vi65]{Vi65} E. B. Vinberg, `The structure of the group of automorphisms of a homogeneous convex cone', 
{\it Trans. Mosc. Math. Soc.} {\bf 14} (1965) 63--93.
\end{thebibliography}
\end{document}